\documentclass[12pt]{article}
\usepackage{amsmath,amssymb,amsthm}

\newtheorem{definition}{Definition}
\newtheorem{theorem}[definition]{Theorem}

\newtheorem{remark}[definition]{Remark}

\newtheorem*{definition*}{Definition}
\newtheorem*{theorem*}{Theorem}
\newtheorem*{proposition*}{Proposition}
\newtheorem*{example*}{Example}
\newtheorem*{exercise*}{Exercise}
\newtheorem*{corollary*}{Corollary}
\newtheorem*{remark*}{Remark}

\usepackage{geometry}
\usepackage[shortlabels, inline]{enumitem}
\usepackage{bm}
\usepackage{graphicx}
\usepackage{float}

\allowdisplaybreaks[2]
\usepackage{color}
\usepackage{mathtools}
\mathtoolsset{showonlyrefs=true}
\usepackage{url}

\begin{document}


\begin{center}
  ~\vspace{20pt}
  
  \Large 
  An alternative formulation of the discrete-time fractional Poisson process

  \vspace{20pt}
  
  \large
  Naohiro Yoshida

  \normalsize
  Department of Economics,
  Keiai University

  1-5-21,
  Anagawa, Inage,
  263-8588,
  Chiba,
  Japan

  E-mail:
  \url{n-yoshida@u-keiai.ac.jp}

\end{center}

  \vspace{20pt}

  

\begin{abstract}
This paper introduces a discrete-time fractional Poisson process defined as a renewal process, where the waiting times follow a discrete Mittag-Leffler distribution. We investigate its fundamental properties by explicitly deriving the probability generating function of the waiting times and the exact probability distribution of the event counts. Through this analysis, we reveal that, unlike its continuous-time counterpart, our renewal-based model is not mathematically equivalent to the process constructed via subordination using the Sibuya distribution.
\end{abstract}

\newpage
\section{Introduction}

Research on the continuous-time fractional Poisson process is said to have originated with the works of \cite{repin2000fractional,jumarie2001fractional,laskin2003fractional}. 
While there are various approaches to defining the process, early studies defining it as a renewal process include, for instance, \cite{mainardi2004fractional}.
Later, this definition was shown to be equivalent to a time-changed standard Poisson process using an inverse stable process as a subordinator in \cite{meerschaert2011fractional}.

The purpose of this paper is to define a renewal process that can be regarded as a discrete-time fractional Poisson process, and to investigate its fundamental properties. By examining the probability generating function of the waiting times during this analysis, we demonstrate that, unlike the continuous-time case, the discrete-time fractional Poisson process defined as a renewal process is not identical to the one defined as a semi-Markov process via subordination (or time-change).

The structure of this paper is as follows. Section \ref{sec2} recalls the fundamentals of discrete fractional calculus. Section \ref{sec3} derives the probability generating function of the waiting times and the probability distribution of the number of events as fundamental properties of the discrete-time fractional Poisson process. Finally, Section \ref{sec4} provides a conclusion and discusses directions for future research.

\section{Fractional Calculus}
\label{sec2}

In this section, we summarize the necessary fractional calculus in a form suitable for this paper. We primarily follow the notation and framework of \cite{anastassiou2010nabla}.

For time $t\in\mathbb{Z}$ and a function $f(t)$, the time difference is denoted by
\begin{equation}
  \nabla f(t):=f(t)-f(t-1).
\end{equation}
The sum over time, acting as the inverse operator of the time difference, is defined for an integer $a\leq t$ as
\begin{equation}
  \nabla_a^{-1}f(t):=\sum_{u=a}^t f(u).
\end{equation}

\subsection{Discrete Power Function}

For $x\in \mathbb{R}\backslash\{0,-1,-2,\dots\}$ and $\alpha \in\mathbb{R}$ such that $x+\alpha \not\in \{0,-1,-2,\dots\}$, we set
$$x^{\overline{\alpha}}:=\frac{\Gamma (x+\alpha)}{\Gamma(x)}.$$
Then, we have
\begin{align}
  \nabla x^{\overline{\alpha}}&=\frac{\Gamma(x+\alpha)}{\Gamma(x)} - \frac{\Gamma(x-1+\alpha)}{\Gamma(x-1)}
  \\
  &=\alpha \frac{\Gamma(x+\alpha -1)}{\Gamma(x)}=\alpha x^{\overline{\alpha -1}}.
\end{align}
Note also that $x^{\overline{0}}=1$.

\subsection{Generalized Binomial Coefficient}

When $x\in \mathbb{R}\backslash\{0,-1,-2,\dots\}$ and $\alpha\in \mathbb{R}\backslash\{-1,-2,\dots\}$ such that $x+\alpha\not\in \{0,-1,-2,\dots\}$, we express the generalized binomial coefficient as
\begin{align}
  \hat{h}_\alpha(x) :&=\frac{\Gamma (x+\alpha )}{\Gamma (\alpha +1) \Gamma (x)}=\frac{x^{\overline{\alpha}}}{\Gamma (\alpha +1)}
  \\
  &= \binom{x+\alpha-1}{x-1}=\binom{x+\alpha-1}{\alpha}.
\end{align}
When $\alpha\not\in \{0,-1,-2,\dots\}$ and $x+\alpha\not\in \{0,-1,-2,\dots\}$, it holds that
$$\hat{h}_{\alpha-1}(x) = \hat{h}_{x-1}(\alpha)=\frac{\Gamma (x+\alpha -1)}{\Gamma (\alpha ) \Gamma (x)}.$$
Note also that $\hat{h}_0(x)=1$ and $\hat{h}_\alpha(1)=1$.

Furthermore, $\hat{h}_\alpha(x)$ are the coefficients of the series expansion of the discrete exponential function $(1-\lambda )^{-x}$ for $|\lambda |<1$. 
Indeed, for $x\in\mathbb{R}\backslash\{0,-1,-2,\dots\}$, this can be seen from
\begin{align}
  (1-\lambda )^{-x}
  &=\sum_{m=0}^\infty \binom{x+m-1}{m}\lambda^m 
  = \sum_{m=0}^\infty \hat{h}_m (x)\lambda^m
  \\
  &=\sum_{n=1}^\infty \hat{h}_{n-1} (x)\lambda^{n-1}
  =\sum_{n=1}^\infty \hat{h}_{x-1} (n)\lambda^{n-1}.
  \label{eq:generating}
\end{align}

Also, for $\alpha \in\mathbb{R}\backslash \{0,-1,-2,\dots\}$, we have
\begin{equation}
  \nabla \hat{h}_{\alpha}(x)
  =\frac{\alpha x^{\overline{\alpha -1}}}{\Gamma (\alpha +1)} =\frac{ x^{\overline{\alpha -1}}}{\Gamma (\alpha )}=\hat{h}_{\alpha -1}(x).
\label{eq:sabun_h}
\end{equation}

\subsection{Fractional Sum}

When $m$ is a positive integer, let $\nabla_a^{-m}$ denote the $m$-fold application of $\nabla_a^{-1}$, i.e., 
$\nabla_a^{-m} f(t) = \sum_{u=a}^t \nabla_a^{-m+1}f(u)$.
From the properties of the generalized binomial coefficient, we see that
\begin{align}
  \nabla_a^{-m} f(t)
&=\sum_{u_1=a}^t \sum_{u_2=a}^{u_1}\cdots \sum_{u_m=a}^{u_{m-1}} f(u_m)
\\
&=\sum_{u=a}^t \hat{h}_{m-1}(t-u+1)f(u).
\end{align}

Generalizing this, for a positive real number $\alpha >0$, $t\in\mathbb{Z}$, and an integer $a\leq t$, we define
$$\nabla_a^{-\alpha} f(t):=\sum_{u=a}^t \hat{h}_{\alpha -1}(t-u+1)f(u).$$

\subsection{Fractional Difference}

When $m$ is a positive integer, let $\nabla^m$ denote the $m$-fold application of $\nabla$, i.e., 
$\nabla^m f(t)=\nabla^{m-1} f(t)-\nabla^{m-1}f(t-1)$.
In this case, by repeatedly applying \eqref{eq:sabun_h}, we have for $\alpha \in\mathbb{R}$
\begin{equation}
  \nabla^m \hat{h}_\alpha (t)=\begin{cases}
  \hat{h}_{\alpha -m}(t), & \text{if }\alpha \not \in \{1,2,\dots,m-1\},
  \\
  0, & \text{if }\alpha \in \{1,2,\dots, m-1\}.
\end{cases}
\end{equation}

For a positive real number $\alpha >0$, the fractional difference $\nabla_a ^\alpha$ is defined using an integer $m$ such that $m-1< \alpha \leq m$ as
\begin{align}
  \nabla_a ^\alpha f(t)&=\nabla^m \nabla_a ^{-(m-\alpha)}f(t)
  \\
  &=\nabla^m \sum_{u=a}^t \hat{h}_{m-\alpha -1}(t-u+1)f(u).
\end{align}
This is known as the Riemann-Liouville type discrete fractional difference, which we adopt in this paper. 
On the other hand, 
${}^C\nabla_a ^\alpha f(t)=\nabla_a^{-(m-\alpha)} \nabla^m f(t)$ 
is called the Caputo type discrete fractional difference.

\subsection{Discrete Mittag-Leffler Function}

We recall the discrete version of the Mittag-Leffler function defined in Definition 5 of \cite{nagai2003discrete}.

For a real number $q>0$, we shall call
\begin{equation}
 F_{q,\lambda}(t):=\sum_{m=0}^\infty \hat{h}_{qm}(t)\lambda^m=\sum_{m=0}^\infty \frac{t^{\overline{qm}}}{\Gamma(qm+1)}\lambda^m 
\end{equation}
the discrete Mittag-Leffler function.
When $q=1$, it becomes $F_{1,\lambda}(t)=(1-\lambda)^{-t}$, which coincides with the discrete exponential function.

We now prepare a convolution formula. 
For $\alpha \in\mathbb{R}\backslash\{0,-1,-2,\dots\}$ and $0<q\leq 1$, from the fact that
\begin{equation}
  (1-\lambda)^{-(\alpha - q + 2)}=\sum_{t=1}^\infty \hat{h}_{\alpha -q+1}(t)\lambda^{t-1}
\end{equation}
and
\begin{align}
  (1-\lambda)^{q-1}(1-\lambda)^{-(\alpha +1)}
  &=\sum_{v=1}^\infty \hat{h}_{-q}(v)\lambda^{v-1}\sum_{u=1}^\infty \hat{h}_{\alpha}(u)\lambda^{u-1}
  \\
  &=\sum_{t=1}^\infty  \sum_{u=1}^t  \hat{h}_{-q}(t-u+1) \hat{h}_{\alpha}(u) \lambda^{t-1},
\end{align}
we obtain by comparing the coefficients for $t\in\{1,2,\dots\}$:
\begin{equation}
 \sum_{u=1}^t  \hat{h}_{-q}(t-u+1) \hat{h}_{\alpha}(u)=\hat{h}_{\alpha -q+1}(t). 
\end{equation}

Furthermore, from this, the following fractional difference formula is obtained:
\begin{align}
  \nabla_1^q \hat{h}_{\alpha}(t)
  &= \nabla \nabla_1^{-(1-q)}\hat{h}_{\alpha}(t)
  \\
  &= \nabla \sum_{u=1}^t \hat{h}_{-q}(t-u+1)\hat{h}_{\alpha}(u)
  \\
  &=\nabla \hat{h}_{\alpha -q+1}(t)
  =\hat{h}_{\alpha -q}(t).
\end{align}
From this, for the fractional difference of the discrete Mittag-Leffler function, we see that
\begin{align}
  \nabla_1^q F_{q,\lambda}(t)
  &=\sum_{m=0}^\infty \nabla_1^q \hat{h}_{qm}(t)\lambda^m
  \\
  &=\sum_{m=0}^\infty  \hat{h}_{q(m-1)}(t)\lambda^m
  \\
  &=\hat{h}_{-q}(t)+\sum_{m=0}^\infty  \hat{h}_{qm}(t)\lambda^{m+1}
  \\
  &=\hat{h}_{-q}(t)+ \lambda F_{q,\lambda}(t).
\end{align}

\section{Discrete Fractional Poisson Process}
\label{sec3}

Let $T_i^q~(i=1,2,\dots)$ be independent and identically distributed (i.i.d.) random variables, sharing the following discrete Mittag-Leffler distribution. That is, supported on $\{1,2,\dots\}$, with $0<\lambda <1,~0<q\leq 1$, and for $t=1,2,\dots$, let
$P(T^q_i> t)=F_{q,-\lambda}(t)$.
Let $S_0=0$ and $S_n^q=T_1^q+T_2^q+\dots+T_n^q~(n=1,2,\dots)$. We define the discrete fractional Poisson process as
$$N^q(t)=\max\{n\geq 0|S^q_n\leq t\}.$$

\subsection{Probability Generating Function of the Waiting Time}

We derive the probability generating function (PGF) of the waiting time from one event to the next.
\begin{theorem}
For $-1\leq z\leq 1$,
 \begin{equation}
   G^q(z)=E[z^{T_1^q}]=\frac{\lambda z}{(1-z)^q + \lambda}.
   \label{eq:PGFofT}
 \end{equation}
\end{theorem}
\begin{proof}
For $u=1,2,\dots$, the probability mass function of $T_1^q$ is
\begin{align}
P(T_i^q=u) &= F_{q,-\lambda}(u-1) - F_{q,-\lambda}(u) \\
&= -\nabla F_{q,-\lambda}(u) \\
&= -\sum_{m=1}^\infty \hat{h}_{qm-1}(u)(-\lambda)^m.
\end{align}
Using this, the generating function $E[z^{T_1^q}]$ can be calculated as follows:
\begin{align}
  E[z^{T_1^q}] &= \sum_{u=1}^\infty z^u P(T_1^q=u) \nonumber \\
  &= - \sum_{u=1}^\infty z^u \left( \sum_{m=1}^\infty \hat{h}_{qm-1}(u)(-\lambda)^m \right)
\\
  &= - \sum_{m=1}^\infty (-\lambda)^m \sum_{u=1}^\infty \hat{h}_{qm-1}(u) z^u
  \label{eq:pgf_step1}
\\
  &= - \sum_{m=1}^\infty (-\lambda)^m z(1-z)^{-qm} \nonumber \\
  &= - z \sum_{m=1}^\infty \left( -\lambda (1-z)^{-q} \right)^m.
\end{align}
Assuming $|z|<1-\lambda^{1/q}$ and calculating the infinite geometric series, we obtain
\begin{align}
  E[z^{T_1^q}] &= - z \frac{-\lambda (1-z)^{-q}}{1 - \left( -\lambda (1-z)^{-q} \right)} \nonumber \\
  &= \frac{\lambda z}{(1-z)^q + \lambda}.
\end{align}
This obtained closed-form expression can be analytically continued to the region $|z|<1$, and by Abel's theorem, taking the limit as $z \to 1$ justifies its value at $z=1$, which yields the assertion.
\end{proof}

In particular, for the standard discrete Poisson process where $q=1$,
$$E[z^{T_1^1}] = \frac{\lambda z}{1 - z + \lambda} = \frac{\frac{\lambda}{1+\lambda} z}{1 - \frac{1}{1+\lambda}z},$$
which completely coincides with the probability generating function of the geometric distribution (for the number of trials up to the first success) with success probability $p = \frac{\lambda}{1+\lambda}$.

Using the probability generating function obtained above, we calculate the expected value $E[T_1^q]$:
\begin{align}
  \frac{d}{dz}G^q(z) &= \frac{\lambda \left( (1-z)^q + \lambda \right) - \lambda z \left( -q(1-z)^{q-1} \right)}{\left( (1-z)^q + \lambda \right)^2} \nonumber \\
  &= \frac{\lambda(1-z)^q + \lambda^2 + \lambda q z (1-z)^{q-1}}{\left( (1-z)^q + \lambda \right)^2}. \label{eq:G_prime}
\end{align}

For $0 < q < 1$, since $q-1 < 0$, we have $\lim_{z \to 1} (1-z)^{q-1} = \infty$. Therefore,
\begin{align}
  E[T_1^q] =\lim_{z \to 1}\frac{d}{dz}G^q(z) =  \infty \quad (0 < q < 1).
\end{align}
It can be seen that in this regime, the waiting time follows a heavy-tailed distribution.

For $q = 1$, since $(1-z)^{q-1} = (1-z)^0 = 1$, the limit does not diverge and can be evaluated as a finite value:
\begin{align}
  E[T_1^1] =\lim_{z \to 1}\frac{d}{dz}G^1(z)=  \frac{\lambda \cdot 0 + \lambda^2 + \lambda \cdot 1}{\lambda^2} = 1 + \frac{1}{\lambda}.
\end{align}
This perfectly matches the expected value $1/p$ of the geometric distribution with success probability $p = \frac{\lambda}{1+\lambda}$.

\begin{remark}
For the continuous-time fractional Poisson process, it is known that the definition as a renewal process is mathematically strictly equivalent to the one defined by subordinating a standard Poisson process using an inverse stable process as a subordinator (see \cite{meerschaert2011fractional}). However, in the discrete-time framework, it is important to note that these two approaches bifurcate into different probabilistic models.
In discrete time, the Sibuya distribution with parameter $0 < q \le 1$ plays the role corresponding to the subordinator in the continuous model (see, e.g., \cite{nichols2018subdiffusive}). When a random variable $X$ follows a Sibuya distribution with parameter $q \in (0,1)$, its probability mass function is defined using generalized binomial coefficients as follows:
\begin{equation}
    P(X = k) = (-1)^{k-1} \binom{q}{k}, \quad (k = 1, 2, 3, \dots).
\end{equation}
Its PGF is given by $S_q(z) = 1 - (1-z)^q$. The PGF of the waiting time for the standard discrete Poisson process ($q=1$) is $G_1(z) = \frac{\lambda z}{1-z+\lambda}$. If we apply subordination to this using the Sibuya distribution, the PGF $G_{sub}(z)$ of the new waiting time is calculated as a composite function:
\begin{align}
  G_{sub}(z) = G_1(S_q(z)) = \frac{\lambda (1 - (1-z)^q)}{(1-z)^q + \lambda}. \label{eq:remark_sub}
\end{align}
Comparing Equation \eqref{eq:remark_sub} with Equation \eqref{eq:PGFofT}, they share the identical denominator $(1-z)^q + \lambda$. This implies that both models share macroscopic fractional characteristics, such as the divergence of the expected value of the waiting time distribution.
However, the structure of the numerators, $1 - (1-z)^q$ and $z$, is distinctly different, indicating that they strictly follow different probabilistic laws on the microscopic discrete steps.
\end{remark}

\subsection{Probability Distribution of the Number of Events}

The probability mass function for the number of events $N^q(t)$ is obtained as follows.

\begin{theorem}
For $n=0,1,2,\dots$ and $t\geq n$,
  \begin{align}
  P(N^q(t)=n) 
  &= \lambda^n \sum_{m=0}^\infty \binom{n+m}{n} (-\lambda)^m 
  \cdot\Big[ \hat{h}_{q(n+m)}(t-n+1) + \lambda \hat{h}_{q(n+m+1)-1}(t-n+1) \Big].
\end{align}
\end{theorem}
\begin{proof}
The event that the counting process $N^q(t)$ is in state $n$ means that exactly $n$ arrivals occur up to time $t$, which can be expressed using the arrival times as follows:
\begin{align}
  P(N(t) = n) = P(S_n \le t) - P(S_{n+1} \le t).
\end{align}
Thus, letting $H^q_n(z) := \sum_{t=0}^\infty P(N(t)=n) z^t$, we have
\begin{align}
  H^q_n(z) &= \sum_{t=0}^\infty P(S_n \le t) z^t - \sum_{t=0}^\infty P(S_{n+1} \le t) z^t \nonumber \\
  &=\frac{E[z^{S_n}]}{1-z} - \frac{E[z^{S_{n+1}}]}{1-z}
  \\
  &= \frac{G^q(z)^n}{1-z} - \frac{G^q(z)^{n+1}}{1-z} \nonumber \\
  &= \frac{G^q(z)^n (1 - G^q(z))}{1-z} 
  \\
  &= \left( \frac{\lambda z}{(1-z)^q + \lambda} \right)^n \frac{(1-z)^{q-1} + \lambda}{(1-z)^q + \lambda}
\\
  &= \lambda^n z^n \left[ (1-z)^{q-1} + \lambda \right] \left( (1-z)^q + \lambda \right)^{-(n+1)}.
\end{align}
Here, by the negative binomial theorem,
\begin{align}
  \left( (1-z)^q + \lambda \right)^{-(n+1)} 
  &= (1-z)^{-q(n+1)} \left( 1 + \lambda (1-z)^{-q} \right)^{-(n+1)} \nonumber \\
  &= (1-z)^{-q(n+1)} \sum_{m=0}^\infty \binom{-(n+1)}{m} \lambda^m (1-z)^{-qm} \nonumber \\
  &= \sum_{m=0}^\infty \binom{n+m}{n} (-\lambda)^m (1-z)^{-q(n+m+1)}.
\end{align}
Therefore,
\begin{equation}
H^q_n(z) 
  = \lambda^n z^n \left[ \sum_{m=0}^\infty \binom{n+m}{n} (-\lambda)^m (1-z)^{-q(n+m)-1}   + \lambda \sum_{m=0}^\infty \binom{n+m}{n} (-\lambda)^m (1-z)^{-q(n+m+1)} \right]. \label{eq:Gk_expanded}
\end{equation}
Here, from \eqref{eq:generating}, for $t \ge n$,
\begin{itemize}
  \item the coefficient of $z^{t-n}$ in the first term $(1-z)^{-(q(n+m)+1)}$ is $\hat{h}_{q(n+m)}(t-n+1)$, and
  \item the coefficient of $z^{t-n}$ in the second term $(1-z)^{-q(n+m+1)}$ is $\hat{h}_{q(n+m+1)-1}(t-n+1)$.
\end{itemize}
Thus, it is found that the coefficient of $z^n$ in $H^q_n(z)$ is
\begin{align}
   \lambda^n \sum_{m=0}^\infty \binom{n+m}{n} (-\lambda)^m 
  \cdot\Big[ \hat{h}_{q(n+m)}(t-n+1) + \lambda \hat{h}_{q(n+m+1)-1}(t-n+1) \Big],
\end{align}
which completes the proof.
\end{proof}

In particular, for $n=0$,
\begin{align*}
  P(N^q(t)=0) 
  &= \sum_{m=0}^\infty (-\lambda)^m \hat{h}_{qm}(t+1) - \sum_{m=0}^\infty (-\lambda)^{m+1} \hat{h}_{q(m+1)-1}(t+1) \\
  &= \hat{h}_0(t+1) + \sum_{m=1}^\infty (-\lambda)^m \left[ \hat{h}_{qm}(t+1) - \hat{h}_{qm-1}(t+1) \right].
\end{align*}
From the difference formula \eqref{eq:sabun_h}, 
we have $\hat{h}_{\alpha}(t+1) - \hat{h}_{\alpha-1}(t+1) = \hat{h}_{\alpha}(t)$, hence
$$P(N^q(t)=0) = \sum_{m=0}^\infty (-\lambda)^m \hat{h}_{qm}(t) = F_{q,-\lambda}(t),$$
which completely coincides with the probability $P(T_1^q > t)$ that the first event occurs after time $t$.

\section{Conclusion}
\label{sec4}

In this paper, we defined a discrete-time fractional Poisson process by constructing a renewal process with the discrete Mittag-Leffler distribution as the distribution of the event waiting times, and derived the probability generating function of the waiting times as well as the probability distribution of the number of events. Through this analysis, we demonstrated that, unlike the continuous-time case, the discrete-time fractional Poisson process in this paper is not identical to the one defined as a semi-Markov process via subordination using the Sibuya distribution.

As a future direction, we aim to investigate long-range dependence, which is one of the main themes of fractional stochastic processes, by characterizing its autocorrelation structure.

Furthermore, this framework can also be applied to ruin probability theory in actuarial science. When claims exhibit a heavy-tailed property, the asymptotic behavior of an insurance company's survival probability can be categorized based on the magnitude relationship between the tail index and the parameter $q$. We hope to present these findings on another occasion.

\end{document}